\documentclass{IEEEtran}
\usepackage{amsmath,amssymb,amsfonts}
\usepackage{graphicx}
\usepackage{textcomp}

\usepackage{cite}
\usepackage{amsmath}
\usepackage{amssymb}
\usepackage{balance}
\usepackage{graphicx}
\usepackage{subfigure}
\usepackage{epstopdf}
\usepackage{amsfonts}
\usepackage[all]{xy}
\usepackage{array,xcolor}
\usepackage{color}

\usepackage[colorlinks=true,      % false: boxed links; true: colored links
			linkcolor=black,      % color of internal links
			citecolor=black,      % color of links to bibliography
			filecolor=black,      % color of file links
			urlcolor=blue ]{hyperref}

\usepackage{caption}
\usepackage{algorithm}
\usepackage{algpseudocode}

\makeatletter
\newcommand{\algorithmfootnote}[2][\footnotesize]{%
  \let\old@algocf@finish\@algocf@finish% Store algorithm finish macro
  \def\@algocf@finish{\old@algocf@finish% Update finish macro to insert "footnote"
    \leavevmode\rlap{\begin{minipage}{\linewidth}
    #1#2
    \end{minipage}}%
  }%
}
\makeatother

\usepackage{threeparttable}
\usepackage{booktabs}
\usepackage{changepage}

\usepackage{multirow}

\newtheorem{proposition}{Proposition}
\newtheorem{remark}{Remark}

\newenvironment{proof}{{\em Proof: }}{\hfill \hspace*{1pt}
\hfill $\blacksquare$}

\newcolumntype{C}[1]{>{\centering}m{#1}}

\def\BibTeX{{\rm B\kern-.05em{\sc i\kern-.025em b}\kern-.08em
    T\kern-.1667em\lower.7ex\hbox{E}\kern-.125emX}}

\newcommand{\tr}{{{\mathsf T}}}

\newcounter{tempEquationCounter}
\newcounter{thisEquationNumber}

\begin{document}

\title{Fast ADMM for sum-of-squares programs\\using partial orthogonality\textsuperscript{$\dagger$}}
\author{Yang Zheng, Giovanni~Fantuzzi, and Antonis Papachristodoulou, \IEEEmembership{Senior Member, IEEE}
%%\thanks{Y. Zheng is supported by the Clarendon Scholarship and the Jason Hu Scholarship, G. Fantuzzi is supported by an EPSRC Doctoral Prize Fellowship, and A. Papachristodoulou is supported by EPSRC Grant EP/M002454/1. }
%%\thanks{Y. Zheng, and A. Papachristodoulou are with Department of Engineering Science at the University of Oxford. (E-mail: \{yang.zheng, antonis\}@eng.ox.ac.uk) }
%%\thanks{G. Fantuzzi is with Department of Aeronautics at Imperial College London (E-mail: gf910@ic.ac.uk.)}
% Combined thanks:
\thanks{Y. Zheng, and A. Papachristodoulou are with Department of Engineering Science at the University of Oxford. (E-mail: \{yang.zheng, antonis\}@eng.ox.ac.uk). G. Fantuzzi is with Department of Aeronautics at Imperial College London (E-mail: gf910@ic.ac.uk.). Y. Zheng is supported by the Clarendon Scholarship and the Jason Hu Scholarship, G. Fantuzzi is supported by an EPSRC Doctoral Prize Fellowship, and A. Papachristodoulou is supported by EPSRC Grant EP/M002454/1.}
\thanks{$^\dagger$This document is an extended version of a homonymous article submitted to \textit{IEEE Trans. Autom. Control}.%, \url{<link-to-published-version-when-available>}.
}
}

\maketitle

\begin{abstract}
%Semidefinite programs (SDPs) arising from sum-of-squares (SOS) programming possess a structural property that we call \emph{partial orthogonality} when they are formulated using the standard monomial basis.
When sum-of-squares (SOS) programs are recast as semidefinite programs (SDPs) using the standard monomial basis, the constraint matrices in the SDP possess a structural property that we call \emph{partial orthogonality}. In this paper, we leverage partial orthogonality to develop a fast first-order method, based on the alternating direction method of multipliers (ADMM), for the solution of the homogeneous self-dual embedding of SDPs describing SOS programs. Precisely, we show how a ``diagonal plus low rank" structure implied by partial orthogonality can be exploited to project efficiently the iterates of a recent ADMM algorithm for generic conic programs onto the set defined by the affine constraints of the SDP. The resulting algorithm, implemented as a new package in the solver CDCS, is tested on a range of large-scale SOS programs arising from constrained polynomial optimization problems and from Lyapunov stability analysis of polynomial dynamical systems. These numerical experiments demonstrate the effectiveness of our approach compared to common state-of-the-art solvers.
\end{abstract}

\begin{IEEEkeywords}
    Sum-of-squares (SOS),
    ADMM,
    large-scale optimization.
\end{IEEEkeywords}

\section{Introduction} \label{Section:Introduction}

Optimizing the coefficients of a polynomial in $n$ variables, subject to a nonnegativity constraint on the entire space $\mathbb{R}^n$ or on a semialgebraic set $\mathcal{S} \subseteq \mathbb{R}^n$ ({\it i.e.}, a set defined by a finite number of polynomial equations and inequalities), is a fundamental problem in many fields. For instance, linear, quadratic and mixed-integer optimization problems can be recast as polynomial optimization problems (POPs) of the form~\cite{lasserre2009moments}
\begin{equation} \label{eq:POP}
\min_{x \in \mathcal{S}}\, p(x),
\end{equation}
where $p(x)$ is a multivariate polynomial and $\mathcal{S} \subseteq \mathbb{R}^n$ is a semialgebraic set. Problem~\eqref{eq:POP} is clearly equivalent to
\begin{equation} \label{eq:POPreformualtion}
    \begin{aligned}
        \max & \quad \gamma \\
        \text{s. t.} & \quad p(x)- \gamma \geq 0 \quad \forall x \in \mathcal{S},
    \end{aligned}
\end{equation}
so POPs of the form~\eqref{eq:POP} can be solved globally if a linear cost function can be optimized subject to polynomial nonnegativity constraints on a semialgebraic set.

Another important example is the construction of a Lyapunov function $V(x)$ to certify that an equilibrium point {$x^*$} of a dynamical system $\tfrac{{\rm d} x(t)}{{\rm d}t} = f(x(t))$ is locally stable. {Taking $x^*=0$ without loss of generality}, given a neighbourhood $\mathcal{D}$ of the origin, local
stability follows if {$V(0)=0$} and 
\begin{subequations}
%\blue{
%\begin{align}
% \label{eq:FindLyapunov_a}
%    \phantom{-}V(x) - x^\tr x&\geq 0 \quad\forall x \in \mathcal{D}, \\
%\label{eq:FindLyapunov_b}
%%- \langle \nabla V(x), f(x) \rangle - x^\tr x &\geq 0 \quad\forall x \in \mathcal{D}.
%- f(x)^\tr \nabla V(x) - x^\tr x &\geq 0 \quad\forall x \in \mathcal{D}.
%\end{align}%}
\begin{align}
 \label{eq:FindLyapunov_a}
    \phantom{-}V(x) &> 0, \quad\forall x \in \mathcal{D} \setminus \{0\}, \\
\label{eq:FindLyapunov_b}
%- \langle \nabla V(x), f(x) \rangle - x^\tr x &\geq 0 \quad\forall x \in \mathcal{D}.
- f(x)^\tr \nabla V(x) &\geq 0, \quad\forall x \in \mathcal{D}.
\end{align}
\end{subequations}
Often, the vector field $f(x)$ is polynomial~\cite{anderson2015advances} and, if one restricts the search to polynomial Lyapunov functions $V(x)$, conditions~\eqref{eq:FindLyapunov_a}-\eqref{eq:FindLyapunov_b} amount to a feasibility problem over nonnegative polynomials.

Testing for nonnegativity, however, is NP-hard for polynomials of degree as low as four~\cite{parrilo2003semidefinite}. This difficulty is often resolved by requiring that the polynomials under consideration are a sum of squares (SOS) of polynomials of lower degree.
In fact, checking for the existence (or lack) of an SOS representation amounts to solving a semidefinite program (SDP)~\cite{parrilo2003semidefinite}. In particular, consider a polynomial of degree $2d$ in $n$ variables,
%Consider a polynomial of degree $2d$ in $n$ variables,
\begin{equation*}
p(x) = \sum_{\alpha \in \mathbb{N}^n, \mid \alpha \mid \leq 2d} p_{\alpha} x_1^{\alpha_1} \dots x_n^{\alpha_n}.
\end{equation*}
{The key observation in~\cite{parrilo2003semidefinite} %(already exploited in the previous work~\cite{powers1998algorithm})
is that an}
SOS representation of $p(x)$ exists if and only if there exists a positive semidefinite matrix $X$
%(known as the {\it Gram matrix} of $p$),
such that
\begin{equation}
\label{e:pxEqualityIntro}
p(x) = v_d(x)^\tr Xv_d(x),
\end{equation}
where
\begin{equation}
\label{eq:CandidateMonomials}
v_d(x) =[ 1,x_1,x_2,\ldots,x_n,x_1^2,x_1x_2,\ldots,x_n^d ]^\tr
\end{equation}
is the vector of monomials of degree no larger than $d$. Upon equating coefficients on both sides of~\eqref{e:pxEqualityIntro}, testing if $p(x)$ is an SOS reduces to a feasibility SDP of the form
\begin{equation}
\label{eq:SDPSOS}
    \begin{aligned}
        \text{find}\quad &X \\
        \text{s. t.} \quad  & \langle B_{\alpha}, X \rangle = p_{\alpha},
        							\quad \alpha\in\mathbb{N}_{2d}^n,\\
            & X \succeq 0,
    \end{aligned}
\end{equation}
where $\mathbb{N}^n_{2d}$ is the set of $n$-dimensional multi-indices with length at most $2d$, $B_\alpha$ are known {symmetric} matrices indexed by such multi-indices (see Section~\ref{Section:Preliminaries} for more details), {and $\langle A,B\rangle = {\rm trace}(AB)$ is the standard Frobenius inner product of two symmetric matrices $A$ and $B$.}

Despite the tremendous impact of SOS techniques in the fields of polynomial optimization~\cite{lasserre2001global} and systems analysis~\cite{papachristodoulou2005tutorial}, the current poor scalability of second-order interior-point algorithms for semidefinite programming prevents the use of SOS methods to solve POPs with many variables, or to analyse dynamical systems with many states. The main issue is that, when the full monomial basis~\eqref{eq:CandidateMonomials} is used, the linear dimension of the matrix $X$ and the number of constraints in~\eqref{eq:SDPSOS} are $N  = { n+d \choose d} $  and $m  = {n+2d \choose 2d}$, respectively, both of which grow quickly as a function of $n$ and $d$.

One strategy to mitigate the computational cost of optimization problems with SOS constraints (hereafter called {\it SOS programs}) is to replace the SDP obtained from the basic formulation outlined above with one that is less expensive to solve using second-order interior-point algorithms.
{Facial reduction techniques~\cite{permenter2014basis}, including the Newton polytope~\cite{reznick1978extremal} and diagonal inconsistency~\cite{lofberg2009pre}, and symmetry reduction strategies~\cite{gatermann2004symmetry} can be utilised to eliminate unnecessary monomials in the basis $v_d(x)$, thereby reducing the size of the positive semidefinite (PSD) matrix variable $X$.} Correlative sparsity~\cite{waki2006sums} can also be exploited to construct sparse SOS representations, wherein a polynomial $p(x)$ is written as a sum of SOS polynomials, each of which  depends only on a subset of the entries of $x$. This enables one to replace the large PSD matrix variable $X$ with a set of smaller PSD matrices, which can be handled more efficiently.  Further computational gains are available if one replaces any PSD constraints---either the original condition $X\succeq 0$ in~\eqref{eq:SDPSOS} or the PSD constraints obtained after applying the aforemention techniques---with the stronger constraints the PSD matrices are diagonally or {scaled-diagonally dominant~\cite{ahmadi2017dsos}.}
%(see \blue{the} DSOS and SDSOS techniques in~\cite{ahmadi2017dsos}).
These conditions can be imposed with linear and second-order cone programming, respectively, and are therefore less computationally expensive. However, while the conservativeness introduced by the requirement of diagonal dominance can be reduced with a basis pursuit algorithm~\cite{ahmadi2015sum}, it cannot generally be removed.%~\cite{ahmadi2017response}.

Another strategy to enable the solution of large SOS programs is to replace the computationally demanding interior-point algorithms with first-order methods, at the expense of reducing the accuracy of the solution.  The design of efficient first-order algorithms for large-scale SDPs has recently received increasing attention:
Wen \emph{et al.} proposed an alternating-direction augmented-Lagrangian method for large-scale dual SDPs~\cite{wen2010alternating}; O'Donoghue \emph{et al.} developed an operator-splitting method to solve the homogeneous self-dual embedding of conic programs~\cite{ODonoghue2016}, which has recently been extended by the authors to exploit aggregate sparsity via chordal decomposition~\cite{ZFPGW2016,zheng2016fast,ZFPGW2017chordal}. Algorithms that specialize in SDPs from SOS programming exist~\cite{henrion2012projection,bertsimas2013accelerated}, but can be applied only to unconstrained POPs---not to constrained POPs of the form~\eqref{eq:POPreformualtion}, nor to the Lyapunov conditions~\eqref{eq:FindLyapunov_a}-\eqref{eq:FindLyapunov_b}. First-order regularization methods have also been applied to large-scale constrained POPs, but without taking into account any problem structure~\cite{nie2012regularization}. Finally, {the sparsity of the matrices $B_\alpha$ in~\eqref{eq:SDPSOS} was exploited in~\cite{Zheng2017Exploiting} to design an operator-splitting algorithm that can solve general large-scale SOS programs, but fails to detect infeasibility (however, recent developments~\cite{BGSBfeasibility2017,liu2017new} may offer a solution for this issue).

One major shortcoming of all but the last of these recent approaches is that they can only be applied to particular classes of SOS programs. For this reason, in this paper we develop a fast first-order algorithm, based on the alternating-direction method of multipliers, for the solution of generic large-scale SOS programs. Our algorithm exploits a particular structural property of SOS programs and can also detect infeasibility. Specifically, our contributions are:

\begin{enumerate}
  \item %At the modeling level,
  We highlight a structural property of SDPs derived from SOS programs using the standard monomial basis: the equality constraints are {\it partially orthogonal}. Notably, the SDPs formulated by {common SOS {modeling toolboxes}~\cite{papachristodoulou2013sostools,henrion2003gloptipoly,lofberg2004yalmip} possess this property.}

  \item %At the computational level,
  We show how partial orthogonality leads to a ``diagonal plus low rank'' matrix structure in the ADMM algorithm of~\cite{ODonoghue2016}, so the matrix inversion lemma can be applied to reduce its computational cost. Precisely, a system of $m \times m$ linear equations to be solved at each iteration can be replaced with a $t\times t$ system, often with $t \ll m$.

  \item %We have implemented the techniques in the MATLAB solver CDCS~\cite{CDCS}, which we refer to as CDCS-sos.
  {We demonstrate the efficiency of our method---available as a new package in the MATLAB solver CDCS~\cite{CDCS}---
  	%We show that partial orthogonality allows for significant computational savings compared to many commonly used
  	compared to many common interior-point solvers
  	%\cite{sturm1999using,toh1999sdpt3,yamashita2003implementation,borchers1999csdp,mosek2010mosek}
  	(SeDuMi~\cite{sturm1999using}, SDPT3~\cite{toh1999sdpt3}, SDPA~\cite{yamashita2003implementation}, CSDP~\cite{borchers1999csdp}, Mosek~\cite{mosek2010mosek})
  	and to the first-order solver SCS~\cite{scs}. Our results on large-scale SOS programs
  	from constrained POPs and Lyapunov stability analysis of nonlinear polynomial systems suggest that the proposed algorithm will enlarge the scale of practical problems that can be handled via SOS techniques.}
\end{enumerate}

The rest of this work is organized as follows. Section~\ref{Section:Preliminaries} briefly reviews SOS programs and their reduction to SDPs. Section~\ref{Section:PartialOrthogonality} {discusses} partial orthogonality in the equality constraints of SDPs arising from SOS programs, while Section~\ref{Section:ADMM} shows how to exploit {it} to facilitate the solution of large-scale SDPs using ADMM. {Sections~\ref{Section:MatrixSOS} and~\ref{Section:weighted-SOS} extend our results to matrix-valued SOS programs and weighted SOS constraints.} Numerical experiments are presented in Section~\ref{Section:Numerical}, and Section~\ref{Section:Conclusion} concludes the paper.

%%%%%%%%%%%%%%%%%%%%%%%%%%%%%%%%%%%
\section{Preliminaries}
\label{Section:Preliminaries}

\subsection{Notation}

%We use standard notation.
The sets of nonnegative integers and real numbers are, respectively, $\mathbb{N}$ and $\mathbb{R}$.
For $x \in \mathbb{R}^n$ and $\alpha \in \mathbb{N}^n$, the monomial $x^{\alpha} = x_1^{\alpha_1}x_2^{\alpha_2}\cdots x_n^{\alpha_n}$ has degree $\vert \alpha \vert := \sum_{i=1}^n \alpha_i$. Given $d \in \mathbb{N}$, we let $\mathbb{N}_d^n=\{\alpha \in \mathbb{N}^n : \vert \alpha\vert \leq d\}$ and $\mathbb{R}[x]_{n,2d}$ be the set of polynomials in $n$ variables with real coefficients of degree $2d$ or less.
A polynomial $p(x) \in \mathbb{R}[x]_{n,2d}$ is a sum-of-squares (SOS) if
$p(x) = \sum_{i=1}^q [ f_i(x) ]^2$,
for some polynomials $f_i \in \mathbb{R}[x]_{n,d}$, $i = 1, \ldots, q$. We denote by $\Sigma[x]_{n,2d}$ the set of SOS polynomials in $\mathbb{R}[x]_{n,2d}$. %The existence of an SOS representation clearly ensures that $p(x) \geq 0, \, \forall x \in \mathbb{R}^n$.
Finally, $\mathbb{S}^n_+$ is the cone of $n\times n$ PSD matrices and $I_{r\times r}$ is the $r \times r$ identity matrix.

\subsection{General SOS programs}
\label{Section:GeneralSOS}

{Consider a vector of optimization variables $u \in \mathbb{R}^t$, a cost vector $w \in \mathbb{R}^t$, and note that any polynomial $p_j(x)\in \mathbb{R}[x]_{n,2d_j}$ whose coefficients depend affinely on $u$ can be written as
$p_j(x) = g^j_0(x) - \sum_{i=1}^t u_ig^j_i(x)$
for a suitable choice of polynomials or monomials $g^j_0,\,\ldots,\,g^j_t \in \mathbb{R}[x]_{n,2d_j}$.
%and tunable SOS polynomials $s_1,\,\ldots,\,s_k\in\Sigma[x]_{n,2d}$ (meaning that their coefficients are optimization variables),
We consider SOS programs written in the standard form}
\begin{equation}
\label{eq:generalSOS-multipleSOS}
\begin{aligned}
\min_{u,\,s_1,\ldots,s_k}\;\; & w^\tr u \\[-1ex]
\text{s. t.} \quad  & s_j(x) = g^j_0(x) - \sum_{i=1}^t u_ig^j_i(x) \; \forall j = 1,\,\ldots,\,k,\\
%& s_1,\,\ldots,\,s_k \in \Sigma[x]_{n,2d}.
& s_j \in \Sigma[x]_{n,2d_j}, \quad j=1 ,\,\ldots,\,k.
\end{aligned}
\end{equation}
{Note that any linear optimization problem with polynomial nonnegativity constraints on fixed semialgebraic sets can be relaxed into an SOS program of the form~\eqref{eq:generalSOS-multipleSOS}.}
For instance, when $\mathcal{S} \equiv \mathbb{R}^n$ problem~\eqref{eq:POPreformualtion} can be relaxed as~\cite{parrilo2003semidefinite}
\begin{equation}
\label{eq:POPsdp}
    \begin{aligned}
        \min_{\gamma,s}\quad & -\gamma \\
        \text{s. t.} \quad  & s(x) = p(x) - \gamma, \\
        						 & s \in\Sigma[x]_{n,2d}. %\;\: \text{is SOS},
    \end{aligned}
\end{equation}
Similarly, the global stability of the origin for a polynomial dynamical system such that $f(0)=0$ may be established by looking for a polynomial Lyapunov function of the form
%$V(x)=g_0(x)-\sum_{i=1}^tu_i g_i(x)$, {where $g_0(0)=\cdots=g_t(0)=0$}.
{$V(x)=-\sum_{i=1}^tu_i g_i(x)$, where $g_1(0)=\cdots=g_t(0)=0$.}
{With $\mathcal{D}\equiv\mathbb{R}^n$, and after subtracting $x^\tr x$ from the left-hand side of~\eqref{eq:FindLyapunov_a} to ensure strict positivity for $x\neq 0$~\cite{papachristodoulou2005tutorial}, suitable values $u_i$ can be found via the SOS feasibility program}
{
\begin{equation}
\label{eq:Lyapunovsdp}
\begin{aligned}
\text{find }\,  & u,\,s_1,\,s_2 \\
\!\text{s.t. }\,
& s_1(x) = - x^\tr x - \sum_{i=1}^t u_i g_i(x),% \;\: \text{is SOS},
\\[-1.5ex]
%& s_2(x) = \left\langle \sum_{i=1}^t u_i  \nabla g_i(x) - \nabla g_0(x), f(x)\right\rangle ,
& s_2(x) = \sum_{i=1}^t u_i  f(x)^\tr\nabla g_i(x),
\\
&  s_1,\,s_2 \in\Sigma[x]_{n,2d}.
\end{aligned}
\end{equation}
}

Sum-of-squares programs arising from polynomial nonnegativity constraints over fixed semialgebraic sets, such as Lasserre's relaxations of constrained POPs~\cite{lasserre2001global} and SOS relaxations of local Lyapunov inequalities~\cite{papachristodoulou2002construction,anderson2015advances}, can also be recast as in~\eqref{eq:generalSOS-multipleSOS} by adding extra polynomials to represent the SOS multipliers introduced after applying {\it Positivstellensatz}~\cite{anderson2015advances}. For example, consider the constrained POP
\begin{equation}\label{E:constrainedPOP}
\begin{aligned}
\min_{x} \quad & p_0(x) \\
\text{s. t.} \quad & p_1(x) \geq 0,\, \ldots,\, p_k(x) \geq 0,
\end{aligned}
\end{equation}
where $p_0,\,\ldots,\,p_k$ are fixed polynomials of degree no greater than $\omega$. The Lasserre relaxation of order $2d \geq \omega$ for~\eqref{E:constrainedPOP} is (see, for example, Chapter 5.3 in~\cite{lasserre2009moments})
\begin{equation}\label{E:constrainedPOPsos}
\begin{aligned}
\min \quad & -\gamma \\
\text{s. t.} \quad & p_0(x) - \gamma = s_0(x) + \sum_{i=1}^k r_i(x) p_i(x), \\
& s_0  \in \Sigma[x]_{n,2d}, \\
& r_j \in \Sigma[x]_{n,2d_j}, \quad j = 1, \ldots, k,
\end{aligned}
\end{equation}
%
%This problem is readily recast in the form~\eqref{eq:generalSOS-multipleSOS} for a suitable set of polynomials $\{g^j_i\}$ after
where $d_j =  \lfloor d-\omega_j/2\rfloor, j=1,\,\ldots,\,k$ and $\omega_j$ is the degree of $p_j(x)$.
Upon introducing extra polynomials $s_1,\,\ldots,\,s_k$ we can consider the equivalent problem
\begin{equation}
\begin{aligned}
\min \quad & -\gamma \\
\text{s. t.} \quad & s_0(x) = p_0(x) - \gamma - \sum_{i=1}^k r_i(x) p_i(x), \\
& s_j(x) = r_j(x), \quad j = 1,\,\ldots,\,k,\\
& s_0\in \Sigma[x]_{n,2d}, \\
& s_j \in \Sigma[x]_{n,2d_j}, \quad j = 1, \ldots,k.
\end{aligned}
\end{equation}
This can be written in the form~\eqref{eq:generalSOS-multipleSOS} for a suitable set of polynomials $\{g^j_i\}$ if the optimization vector $u$ lists the scalar $\gamma$ and the coefficients of the tunable polyomials $r_1,\,\ldots,\,r_k$. A similar argument holds for linear optimization problems with polynomial inequalities on semialgebraic domains, such as the feasibility problems arising from local Lyapunov stability analysis. % (in this context, application of the Positivstellensatz is often referred to as the ``S procedure'').

Of course, while the introduction of extra polynomials allows one to reformulate problem~\eqref{E:constrainedPOPsos} in the framework given by~\eqref{eq:generalSOS-multipleSOS}, it is undesirable in practice because it increases the number of optimization variables. In Section~\ref{Section:weighted-SOS} we show how problems with weighted SOS constraints such as~\eqref{E:constrainedPOPsos} can be handled directly with no need for extra optimization variables. Before that, however, we consider the standard form~\eqref{eq:generalSOS-multipleSOS} as a general framework for SOS programming. To simplify the exposition, instead of~\eqref{eq:generalSOS-multipleSOS}, we will consider the basic SOS program
\begin{equation}
\label{eq:generalSOS}
    \begin{aligned}
        \min_{u,\,s}\quad & w^\tr u \\[-1.5ex]
        \text{s. t.} \quad  & s(x) = g_0(x) - \sum_{i=1}^t u_ig_i(x), \\
        						 & s \in \Sigma[x]_{n,2d}.
    \end{aligned}
\end{equation}
All of our results from Sections~\ref{Section:PartialOrthogonality} and~\ref{Section:ADMM} extend to~\eqref{eq:generalSOS-multipleSOS} when $k>1$, because each of $s_1,\,\ldots,\,s_k$ enters one and only one equality constraint, as well as to more general SOS programs with additional linear equality, inequality, or conic constraints on $u$.

%\end{remark}
\subsection{SDP formulation}

The SOS program~\eqref{eq:generalSOS} can be converted into an SDP upon fixing a basis to represent the SOS polynomial variables. The simplest and most common choice to represent a degree-$2d$ SOS polynomial is the basis $v_d(x)$ of monomials of degree no greater than $d$, defined in~\eqref{eq:CandidateMonomials}. {As discussed in~\cite{parrilo2003semidefinite} and~\cite{powers1998algorithm}}, the polynomial $s(x)$  in~\eqref{eq:generalSOS} is SOS if and only if
\begin{equation} \label{eq:SOSbasicForm}
  s(x) = v_d(x)^\tr Xv_d(x) = \left\langle X, v_d(x)v_d(x)^\tr  \right\rangle,
  \,\,X \succeq 0.
\end{equation}
Let $B_{\alpha}$ be the $0/1$ indicator matrix for the monomial $x^{\alpha}$ in the outer product matrix $v_d(x)v_d(x)^\tr $, \emph{i.e.},
\begin{equation} \label{eq:coeff}
    (B_\alpha)_{\beta,\gamma} =
    \begin{cases}
    		1 &\text{if } \beta + \gamma = \alpha\\
    		0 &\text{otherwise},
    	\end{cases}
\end{equation}
where the natural ordering of multi-indices $\beta, \gamma \in \mathbb{N}^n_d$ is used to index the entries of $B_{\alpha}$. Then,
\begin{equation} \label{E:MatrixBasis}
    v_d(x)v_d(x)^\tr  = \sum_{\alpha \in \mathbb{N}^n_{2d}} B_{\alpha} x^{\alpha}.
\end{equation}
Upon writing
$g_i(x) = \sum_{\alpha \in \mathbb{N}^n_{2d}} g_{i,\alpha} x^{\alpha}$
for each $i = 0, 1, \ldots, t$, and representing $s(x)$ as in~\eqref{eq:SOSbasicForm}, the equality constraint in~\eqref{eq:generalSOS} becomes
\begin{align}
    \sum_{\alpha \in \mathbb{N}^n_{2d}}
    		\left(g_{0,\alpha} - \sum_{i=1}^t u_i g_{i,\alpha} \right) x^{\alpha}
    & = \left\langle X, v_d(x)v_d(x)^\tr  \right\rangle
    \nonumber \\
    & = \sum_{\alpha \in \mathbb{N}^n_{2d}} \left\langle B_{\alpha}, X \right\rangle x^{\alpha}.
\end{align}

Matching the coefficients on both sides yields
\begin{equation}\label{eq:Coefficient}
  g_{0,\alpha} - \sum_{i=1}^t u_i g_{i,\alpha} = \langle B_{\alpha}, X \rangle,
  \quad \forall \alpha \in \mathbb{N}^n_{2d}.
\end{equation}
We refer to~\eqref{eq:Coefficient} as the \emph{coefficient matching conditions}~\cite{Zheng2017Exploiting}.
The SOS program~\eqref{eq:generalSOS} is then equivalent to the SDP
\begin{equation}
\label{eq:generalSOSSDP}
    \begin{aligned}
        \min_{u}\;\, & w^\tr u \\[-1ex]
        \text{s. t.} \;\, &\langle B_{\alpha}, X \rangle +
        \sum_{i=1}^t u_i g_{i,\alpha} =  g_{0,\alpha} \;\,  \forall \alpha \in \mathbb{N}^n_{2d},  \\
        						 & X \succeq 0.
    \end{aligned}
\end{equation}
%{Note that the process of using the Gram matrix representation to derive~\eqref{eq:generalSOSSDP} was discussed in~\cite{powers1998algorithm}.}
As already mentioned in Section~\ref{Section:Introduction}, when the full monomial basis $v_d(x)$ is used to formulate the SDP~\eqref{eq:generalSOSSDP}, the size of  $X$ and the number of constraints are, respectively,
$N = {n+d \choose d}$ and $m  = { n+2d \choose 2d}$. The size of SDP~\eqref{eq:generalSOSSDP} may be reduced (often significantly) by eliminating redundant monomials in $v_d(x)$ based on the structure of the polynomials $g_0(x),\,\ldots,\,g_t(x)$}; the interested reader is referred to Refs.~\cite{reznick1978extremal,lofberg2009pre,gatermann2004symmetry,permenter2014basis}.

\section{Partial orthogonality in SOS programs}
\label{Section:PartialOrthogonality}

For simplicity, we re-index the coefficient matching conditions~\eqref{eq:Coefficient} using integers $i = 1, \ldots, m$ instead of the multi-indices $\alpha$. Let ${\rm vec}: \mathbb{S}^N \to \mathbb{R}^{N^2}$ map a matrix to the stack of its columns and define
$A_1\in\mathbb{R}^{m \times t}$ and $A_2\in \mathbb{R}^{m \times N^2}$ as
\begin{align}
\label{E:VecMatrixDef}
    A_1 &:=  \begin{bmatrix} g_{1,1} & \cdots & g_{t,1} \\
                                \vdots & \ddots & \vdots \\
                                g_{1,m} & \cdots & g_{t,m} \end{bmatrix} ,
	&
    A_2 &:= \begin{bmatrix}
    		{\rm vec}(B_1)^\tr  \\ \vdots \\ {\rm vec}(B_m)^\tr  \end{bmatrix}.
\end{align}
In other words, $A_1$ collects the coefficients of polynomials $g_i(x)$ column-wise, and $A_2$ lists the vectorized {matrices $B_\alpha$ (after re-indexing) in a row-wise fashion}.
Finally, let $\mathcal{S}_+$ be the vectorized positive semidefinite cone, such that ${\rm vec}(X)\in \mathcal{S}_+$ if and only if $X\succeq 0$, and define
\begin{subequations}
\begin{align}
A &:= \left[ A_1,\; A_2\right] \in \mathbb{R}^{m\times (t+N^2)},\\
b &:= \left[ g_{0,1},\,\ldots,\,g_{0,m}\right]^\tr \in \mathbb{R}^m, \label{e:b-vector}\\
c &:= \left[ w^\tr ,\, 0,\,\ldots,\,0 \right]^\tr  \in \mathbb{R}^{t+N^2},\\
\xi &:= \left[ u^\tr , \, {\rm vec}(X)^\tr  \right]^\tr  \in \mathbb{R}^{t+N^2},\\
\mathcal{K} &:= \mathbb{R}^t \times \mathcal{S}_+\,.
\end{align}
\end{subequations}
%
%We can then rewrite~\eqref{eq:generalSOSSDP} as the primal-form conic problem
{Then, noticing from the definition of the trace inner product of matrices that $\langle B_m, X \rangle={\rm vec}(B_m)^\tr {\rm vec}(X)$,} we can rewrite~\eqref{eq:generalSOSSDP} as the primal-form conic {program}
\begin{equation}
\label{eq:generalSDP}
    \begin{aligned}
        \min_{\xi}\quad & c^\tr \xi \\
        \text{s. t.} \quad  & A\xi = b,  \\
        						 & \xi \in \mathcal{K}.
    \end{aligned}
\end{equation}

The key observation at this stage is that the rows of the constraint matrix $A$ are {\it partially orthogonal}. {We show this next, assuming without loss of generality that $t < m$; in fact, very often $t \ll m$ in practice (cf. Tables~\ref{T:TimePOP} and~\ref{T:TimeLyapunov} in Section~\ref{Section:Numerical}).}
\begin{proposition}
\label{theo:orthogonality}
    Let $A=[A_1,\;A_2]$ be the constraint matrix in the conic formulation~\eqref{E:VecMatrixDef} of a SOS program modeled using the monomial basis. The $m \times m$ matrix $AA^\tr $ is of the ``diagonal plus low rank'' form. Precisely, $D:=A_2A_2^\tr $ is diagonal and
$AA^\tr  = D+A_1A_1^\tr $.
\end{proposition}

\begin{proof}
The definition of $A$ implies $AA^\tr  = A_1A_1^\tr  + A_2A_2^\tr $, so we need to show that $A_2A_2^\tr $ is diagonal. This follows from the definition~\eqref{eq:coeff} of the matrices $B_\alpha$: if an entry of $B_\alpha$ is nonzero, the same entry in $B_\beta$, $\alpha\neq \beta$, must be zero. Upon re-indexing the matrices using integers $i=1,\,\ldots,\,m$ as explained above and letting $n_i$ be the number of nonzero entries in $B_i$, it is clear that
$ {\rm vec}(B_i)^\tr  {\rm vec}(B_j) = n_{i}$ if $i = j$, and zero otherwise. Thus, $A_2 A_2^\tr  = {\rm diag}(n_1,\,\ldots,\,n_m)$.
\end{proof}

In essence, Proposition~\ref{theo:orthogonality} states that the constraint sub-matrices corresponding to the matrix $X$ in the SOS decomposition~\eqref{eq:SOSbasicForm} are orthogonal. This fact is a basic {structural} property for {\it any} SOS program formulated using the usual monomial basis. It is not difficult to check that Proposition~\ref{theo:orthogonality} also holds when the full monomial basis $v_d(x)$ is reduced using any of the techniques implemented in any of the modeling toolboxes~\cite{papachristodoulou2013sostools,henrion2003gloptipoly,lofberg2004yalmip}. %SOSTOOLS~\cite{papachristodoulou2013sostools}, GloptiPoly~\cite{henrion2003gloptipoly} and YALMIP~\cite{lofberg2004yalmip}.

\begin{remark}
    In general, the product $A_1A_1^\tr $ has no particular structure, and $AA^\tr $ is not diagonal except for very special problem classes. For example, Figure~\ref{F:SP_Demo2} illustrates the sparsity pattern of $AA^\tr $, $A_1A_1^\tr $, and $A_2A_2^\tr $ for {\small \texttt{sosdemo2}} in SOSTOOLS~\cite{papachristodoulou2013sostools}, an SOS formulation of a Lyapunov function search: $A_2A_2^\tr $ is diagonal, but $A_1A_1^\tr $ and $AA^\tr $ are not. This makes the algorithms proposed in~\cite{bertsimas2013accelerated,henrion2012projection} inapplicable, as they require that $AA^\tr$ is diagonal.
\end{remark}

\begin{figure}%[t]%[!b]
    \centering
    \setlength{\abovecaptionskip}{0pt}
    \setlength{\belowcaptionskip}{0em}
    \subfigure[]
    {
		\includegraphics[width=0.28\columnwidth]{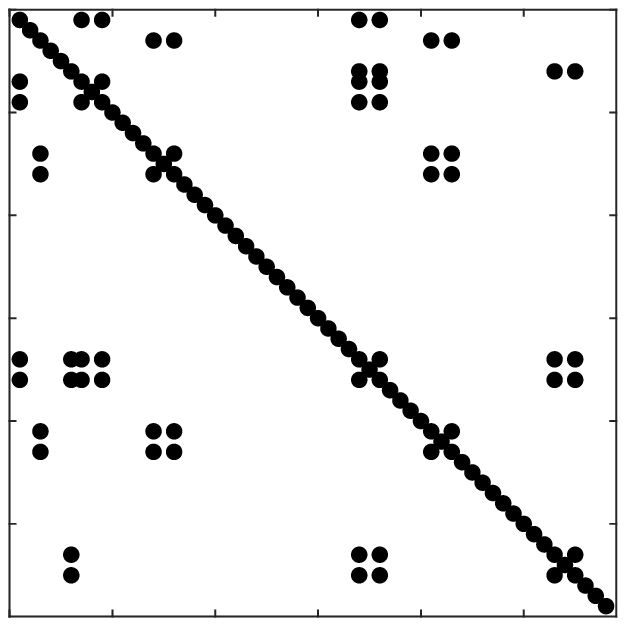}
    }\hspace{2 mm}
    \subfigure[]
    {
		\includegraphics[width=0.28\columnwidth]{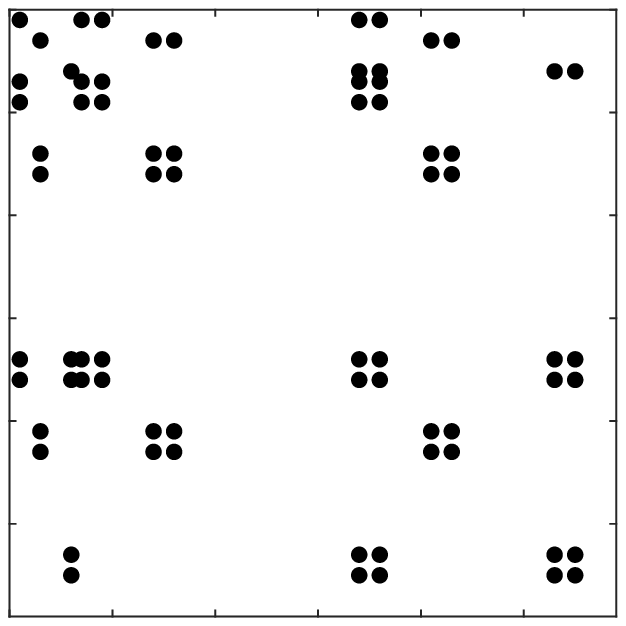}
    }\hspace{2 mm}
    \subfigure[ ]
    {
		\includegraphics[width=0.28\columnwidth]{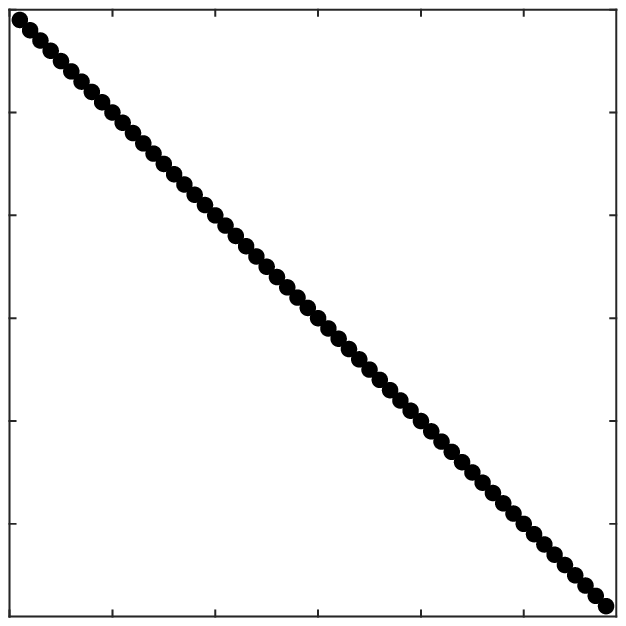}
    }
	\caption{Sparsity patterns for (a) $AA^\tr $, (b) $A_1 A_1^\tr $, and (c) $A_2 A_2^\tr $ for problem \texttt{sosdemo2} in SOSTOOLS~\cite{papachristodoulou2013sostools}.}
    \label{F:SP_Demo2}
    \vspace{-3mm}
\end{figure}

\begin{remark}
    Using the monomial basis to formulate the coefficient matching conditions~\eqref{eq:Coefficient} makes the matrix $A$ sparse, because only a small subset of entries of the matrix $v_d(x)v_d(x)^\tr $ are equal to a given monomial $x^\alpha$. In particular, the density of the nonzero entries of $A_2$ is $\mathcal{O}(n^{-2d})$~\cite{Zheng2017Exploiting}. However, the aggregate sparsity pattern of SDP~\eqref{eq:generalSDP} is dense, so methods that exploit aggregate sparsity in SDPs~\cite{ZFPGW2016,ZFPGW2017chordal,zheng2016fast,fukuda2001exploiting} are not useful for general SOS programs.
    \vspace{-1.5mm}
\end{remark}

%%%%%%%%%%%%%%%%%%%%%%%%%%%%%%%%%%%%%%%%%%%%%%%%%%%%%%%%%%
\section{A fast ADMM-based algorithm}\label{Section:ADMM}

{Partial orthogonality of the constraint matrix $A$ in conic programs of the form~\eqref{eq:generalSDP} allows for the extension of a first-order, ADMM-based method proposed in~\cite{ODonoghue2016}. To make this paper self-contained, we summarize this algorithm first.}

\subsection{The ADMM algorithm}

The algorithm in~\cite{ODonoghue2016} solves the homogeneous self-dual embedding~\cite{ye1994nl} of the conic program~\eqref{eq:generalSDP} and its dual,
\begin{equation}
\label{eq:generalSDPdual}
    \begin{aligned}
        \max_{y,z}\quad & b^\tr y \\
        \text{s. t.} \quad  & A^\tr y+z = c.  \\
        						 & z \in \mathcal{K}^*,
    \end{aligned}
\end{equation}
where the cone $ \mathcal{K}^*$ is the dual of $ \mathcal{K}$. When strong duality holds, optimal solutions for~\eqref{eq:generalSDP} and~\eqref{eq:generalSDPdual} or a certificate of primal or dual infeasibility can be recovered from a nonzero solution of the homogeneous linear system
\begin{equation} \label{eq:HSDE}
    \begin{bmatrix} z \\ s \\ \kappa\end{bmatrix} = \begin{bmatrix} 0 & -A^\tr  & c \\
                                                                    A & 0 & -b \\
                                                                    -c^\tr  & b^\tr  & 0 \end{bmatrix}  \begin{bmatrix} \xi \\ y \\ \tau\end{bmatrix},
\end{equation}
provided that it also satisfies $(\xi,y,\tau) \in \mathcal{K} \times \mathbb{R}^{m} \times \mathbb{R}_{+}$ and  $(z,s,\kappa) \in \mathcal{K}^* \times \{0\}^{m} \times \mathbb{R}_{+}$. The interested reader is referred to~\cite{ODonoghue2016} and references therein for more details. Consequently, upon defining
\begin{align} \label{eq:DefinitionQ}
  u &:= \begin{bmatrix}
    \xi \\ y \\ \tau
  \end{bmatrix},
  &
  v &:= \begin{bmatrix}
    z \\ s \\ \kappa
  \end{bmatrix},
  &
  Q &:= \begin{bmatrix} 0 & -A^\tr  & c \\
                                    A & 0 & -b \\
                                    -c^\tr  & b^\tr  & 0 \end{bmatrix},
\end{align}
and introducing the cones
$\mathcal{C} := \mathcal{K} \times \mathbb{R}^{m} \times \mathbb{R}_{+}$ and $\mathcal{C}^* := \mathcal{K}^* \times \{0\}^{m} \times \mathbb{R}_{+}$ to ease notation, a primal-dual optimal point for problems \eqref{eq:generalSDP} and \eqref{eq:generalSDPdual} or a certificate of infeasibility can be computed from a nonzero solution of the homogeneous self-dual feasibility problem
\begin{equation} \label{E:ADMMForm}
\begin{aligned}
\text{find} \qquad &(u,v)\\
\text{s. t.} \quad &v = Qu,\\
			      &(u,v) \in \mathcal{C} \times \mathcal{C}^*.
\end{aligned}
\end{equation}

It was shown in~\cite{ODonoghue2016} that~\eqref{E:ADMMForm} can be solved using a simplified version of the classical ADMM algorithm (see \emph{e.g.},~\cite{boyd2011distributed}), whose $k$-th iteration consists of the following three steps ($\mathbb{P}_{\mathcal{C}}$ denotes projection onto the cone $\mathcal C$, and the superscript $(k)$ indicates the value of a variable after the $k$-th iteration):
\begin{subequations} \label{eq:ADMMSteps}
    \begin{align}
    \label{eq:HsdeStep1}
      \hat{u}^{(k)} &= (I+Q)^{-1}\left(u^{(k-1)}+ v^{(k-1)}\right), \\
    \label{eq:HsdeStep2}
      u^{(k)} &= \mathbb{P}_{\mathcal{C}}\left(\hat{u}^{(k)}-v^{(k-1)}\right),\\
    \label{eq:HsdeStep3}
      v^{(k)} &= v^{(k-1)} - \hat{u}^{(k)} + u^{(k)}.
    \end{align}
\end{subequations}
{Practical implementations of the algorithm rely on being able to carry out these steps at moderate computational cost. We next show that partial orthogonality allows  for an efficient implementation of~\eqref{eq:HsdeStep1} when~\eqref{E:ADMMForm} represents an SOS program.}

\subsection{Application to SOS programming}

Each iteration of the ADMM algorithm requires: a projection onto a linear subspace in~\eqref{eq:HsdeStep1} through the solution of a linear system with coefficient matrix $I + Q$; a projection onto the cone $\mathcal{C}$ in~\eqref{eq:HsdeStep2}; and the inexpensive step~\eqref{eq:HsdeStep3}. The conic projection~\eqref{eq:HsdeStep2} can be computed efficiently when the cone size is not too large. On the other hand, $Q \in \mathbb{S}^{t+N^2+m+1}$ and $m = \mathcal{O}(n^{2d})$ is extremely large in SDPs arising from SOS programs. For instance, an SOS program with polynomials of degree $2d=6$ in $n=16$ variables has a PSD variable of size $N=969$ and $m=74\,613$ equality constraints. {This makes step \eqref{eq:HsdeStep1} computationally expensive not only if $I+Q$ is factorized directly, but also when applying the strategies proposed in~\cite{ODonoghue2016}.} Fortunately, $Q$ is highly structured and, in the context of SOS programming, the block-entry $A$ {has partially orthogonal rows} (cf. Propositions~\ref{theo:orthogonality} and~\ref{theo:orthogonalityMatrixSOS}). As we will now show, {these properties can be taken advantage of to achieve substantial computational savings.}

To show how partial orthogonality can be exploited, we begin by noticing that~\eqref{eq:HsdeStep1} requires the solution of a linear system of equations of the form
\begin{equation} \label{eq:LinearSystem}
\begin{bmatrix}
I & -A^\tr  & c \\
A & I & -b \\
-c^\tr  & b^\tr  & 1
\end{bmatrix}
\begin{bmatrix}
\hat{u}_1 \\ \hat{u}_2 \\ \hat{u}_3
\end{bmatrix}
=
\begin{bmatrix}
\omega_1 \\ \omega_2 \\ \omega_3
\end{bmatrix}.
\end{equation}

After letting
$$
%\begin{aligned}
M := \begin{bmatrix}I & -A^\tr  \\A & I\end{bmatrix}, \quad \zeta := \begin{bmatrix} c\\ -b \end{bmatrix},
%\end{aligned}
$$
%$\zeta := [c^\tr ,\,-b^\tr ]^\tr $,
%$\zeta := \begin{bmatrix}\begin{smallmatrix} c\\ -b \end{smallmatrix}\end{bmatrix}$
%$M := \begin{bmatrix}\begin{smallmatrix}I & -A^\tr  \\A & I\end{smallmatrix}\end{bmatrix}$,
%%$\zeta := [c^\tr ,\,-b^\tr ]^\tr $,
%$\zeta := \begin{bmatrix}\begin{smallmatrix} c\\ -b \end{smallmatrix}\end{bmatrix}$
%\begin{equation*}
%M := \begin{bmatrix}I & -A^\tr  \\A & I\end{bmatrix}, \quad
%\zeta := \begin{bmatrix}{c} \\ {-b} \end{bmatrix},
%\end{equation*}
%
and eliminating $\hat{u}_3$ from the first and second block-equations in~\eqref{eq:LinearSystem} we obtain
\begin{subequations}
\begin{align}
(M+\zeta\zeta^\tr ) \begin{bmatrix}\hat{u}_1 \\ \hat{u}_2  \end{bmatrix} &= \begin{bmatrix}\omega_1 \\ \omega_2  \end{bmatrix} - \omega_3 \zeta. \label{E:MatrixInverse} \\
\hat{u}_3 &= \omega_3 + {c}^\tr  \hat{u}_1 - {b}^\tr \hat{u}_2. \label{E:FirstBlockElimination}
\end{align}
\end{subequations}
Applying the matrix inversion lemma~\cite{boyd2004convex} to \eqref{E:MatrixInverse} yields
\begin{equation} \label{E:MatrixInverseResult}
    \begin{bmatrix}\hat{u}_1 \\ \hat{u}_2  \end{bmatrix}  =
    \left[
    I -
    \frac{ (M^{-1} \zeta)\zeta^\tr  }{1 + \zeta^\tr  (M^{-1} \zeta)}
    \right]M^{-1}
    \begin{bmatrix}\omega_1 -c\omega_3 \\ \omega_2+b\omega_3  \end{bmatrix} .
\end{equation}
Note that the first matrix on the right-hand side of~\eqref{E:MatrixInverseResult} only depends on problem data, and can be computed before iterating the ADMM algorithm. Consequently, all that is left to do at each iteration is to solve a linear system of equations of the form
\begin{equation} \label{E:SecondBlockElimination}
    \begin{bmatrix}
    I & -{A}^\tr  \\
    {A} & I
    \end{bmatrix}
    \begin{bmatrix}
    \sigma_1 \\ \sigma_2
    \end{bmatrix}
    = \begin{bmatrix}
    \hat{\omega}_1 \\ \hat{\omega}_2
    \end{bmatrix}.
\end{equation}

Eliminating $\sigma_1$ from the second block-equation in~\eqref{E:SecondBlockElimination} gives
\vspace{-0.5cm}
\begin{subequations} \label{eq:Sigma1Sigma2}
    \begin{align}
        \sigma_1 &= \hat{\omega}_1 + {A}^\tr  \sigma_2,\\
        \label{eq:SecBlkEliResult}
        (I + {A}{A}^\tr ) \sigma_2 &= -{A} \hat{\omega}_1 + \hat{\omega}_2.
    \end{align}
\end{subequations}
It is at this stage that partial orthogonality comes into play: by Propositions~\ref{theo:orthogonality} and~\ref{theo:orthogonalityMatrixSOS}, there exists a diagonal matrix $P$ such that $I+AA^\tr  = I+A_1A_1^\tr +A_2A_2^\tr  = P+A_1A_1^\tr $. Recalling from Section~\ref{Section:PartialOrthogonality} that $A_1\in\mathbb{R}^{m\times t}$ with $t\ll m$ for typical SOS programs (\emph{e.g.}, $t=3$ and $m = 58$ for problem {\small\texttt{sosdemo2}} in SOSTOOLS), it is therefore convenient to apply the matrix inversion lemma to~\eqref{eq:SecBlkEliResult} and write
\begin{align*}
    (I + AA^\tr )^{-1} & = (P + A_1A_1^\tr )^{-1} \\
    &= P^{-1} - P^{-1}A_1(I+A_1^\tr P^{-1}A_1)^{-1}A_1^\tr P^{-1}.
\end{align*}

Since $P$ is diagonal, its inverse is immediately computed. Then, $\sigma_1$ and $\sigma_2$ in~\eqref{eq:Sigma1Sigma2} are found upon solving a $t\times t$ linear system with coefficient matrix
\begin{equation} \label{E:HSDEFactorizationFinal}
    I+A_1^\tr P^{-1}A_1 \in \mathbb{S}^{t},
\end{equation}
plus relatively inexpensive matrix-vector, vector-vector, and scalar-vector operations. Moreover, since the matrix $I+A_1^\tr P^{-1}A_1$ depends only on the problem data and does not change at each iteration, its preferred factorization %(or inverse, \red{if $t$ is smaller than the expected number of iterations of the algorithm})
can be cached before iterating steps~\eqref{eq:HsdeStep1}-\eqref{eq:HsdeStep3}.
Once $\sigma_1$ and $\sigma_2$ have been computed, the solution of~\eqref{eq:LinearSystem} can be recovered using vector-vector and scalar-vector operations.

{
\begin{remark} \label{r:flops}
    %The block-elimination method described above coincides with that in~\cite{ODonoghue2016} up to equation~\eqref{E:SecondBlockElimination}.
    In~\cite{ODonoghue2016}, system~\eqref{E:SecondBlockElimination} is solved either through a ``direct'' method based on a cached $LDL^\tr$ factorization, or by applying the ``indirect'' conjugate-gradient (CG) method to~\eqref{eq:SecBlkEliResult}. Both these approaches are reasonably efficient, but exploiting partial orthogonality is advantageous because only a smaller linear system with size $t\times t$ need be solved, with $t \leq m$ and typically $t\ll m$. As shown in the Appendix, when sparsity is ignored, each iteration of our method to solve~\eqref{E:SecondBlockElimination} requires $\mathcal{O}(t^2+mN^2+mt)$ floating-point operations (flops), compared to $\mathcal{O}((t+N^2+m)^2)$ flops for the ``direct'' method of~\cite{ODonoghue2016} and $\mathcal{O}(n_{\rm cg} m^2+mN^2+mt)$ flops for the ``indirect'' method with $n_{\rm cg}$ CG iterations. Of course, practical implementations of the methods of~\cite{ODonoghue2016} exploit sparsity and have a much lower complexity than stated, but the results in Section~\ref{Section:Numerical} confirm that the strategy outlined in this work remains more efficient.
    %the preferred factorization of $I+AA^\tr$ into upper/lower triangular matrices can also be cached before starting the ADMM iterations, {exploiting partial orthogonality is still advantageous}: at each iteration, the upper and lower triangular systems to be solved have size $t\times t$ instead of $m\times m$, and typically $t\ll m$.
\end{remark}
}

%%%%%%%%%%%%%%%%%%%%%%%%%%%%%%%%%%%%%%%%%%%%%%%%%%%%%%%%%%%%%
\section{{Matrix-valued SOS programs}}\label{Section:MatrixSOS}

{Up to this point we have discussed partial orthogonality for scalar-valued SOS programs, but our results and the algorithm proposed in Section~\ref{Section:ADMM} extend also to the matrix-valued case.

Given symmetric matrices $C_{\alpha} \in \mathbb{S}^r$, we say that the symmetric matrix-valued polynomial
\begin{equation*} %\label{E:matrixPoly}
P(x) := \sum_{\alpha  \in \mathbb{N}^n_{2d} } C_{\alpha} x^{\alpha}
\end{equation*}
is an SOS matrix if there exits a $q\times r$ polynomial matrix $H(x)$ such that
$P(x) = H(x)^\tr H(x)$.} Clearly, an SOS matrix is positive semidefinite for all $x \in \mathbb{R}^n$. It is known~\cite{scherer2006matrix} that $P(x)$ is an SOS matrix if and only if there exists a PSD matrix $Y \in \mathbb{S}^l_+$ with $l= r \times {n+d \choose d}$ such that
\begin{equation} \label{E:matrixSOS}
P(x) = \left(I_r \otimes v_d(x)\right)^\tr  Y \left(I_r \otimes v_d(x)\right).
\end{equation}
Similar to~\eqref{eq:generalSOS}, we consider the matrix-valued SOS program
\begin{equation} \label{E:matrixSOSprogram}
\begin{aligned}
\min_{u}\quad & w^\tr u \\[-1ex]
\text{s. t.} \quad  & P(x) = P_0(x) - \sum_{h=1}^t u_hP_h(x),
\\ &P(x) \text{ is SOS},
\end{aligned}
\end{equation}
where $P_0(x),\,\ldots,\,P_t(x)$ are given symmetric polynomial matrices. {Using~\eqref{E:matrixSOS}, matching coefficients, and vectorizing,} the matrix-valued SOS program~\eqref{E:matrixSOSprogram} can be recast as a conic {program} of standard primal-form~\eqref{eq:generalSDP}, for which the following proposition holds.

\begin{proposition} \label{theo:orthogonalityMatrixSOS}
	The constraint matrix $A$ in the conic program formulation of the matrix-valued SOS problem~\eqref{E:matrixSOSprogram} has partially orthogonal rows, \emph{i.e.}, it can be partitioned into  $A = \begin{bmatrix} A_1 \, A_2\end{bmatrix}$ such that $A_2A_2^\tr $ is diagonal.
\end{proposition}

\begin{proof}
    First, introduce matrices $C_{\alpha}(u)$, affinely dependent on $u$, such that
    $$
    P_0(x) - \sum_{h=1}^t u_hP_h(x) =
    \sum_{\alpha \in \mathbb{N}^n_{2d}}
    %\left(C_{0,\alpha} - \sum_{h=1}^tu_h C_{h,\alpha}\right)
    C_\alpha(u)\, x^{\alpha}.
    $$
    By virtue of~\eqref{E:MatrixBasis}, the SOS representation~\eqref{E:matrixSOS} of $P(x)$ can be written as
%    \begin{equation*}
%        \begin{aligned}
%        &\sum_{\alpha \in \mathbb{N}^n_{2d}} \left(C_{0,\alpha} - \sum_{h=1}^tu_h C_{h,\alpha}\right)x^{\alpha} \\
%        = &\sum_{\alpha \in \mathbb{N}^n_{2d}} \begin{bmatrix} \langle Y_{11}, B_{\alpha} \rangle &  \ldots & \langle Y_{r1}, B_{\alpha} \rangle \\
%                            \vdots & \ddots & \vdots \\
%                            \langle Y_{r1}, B_{\alpha} \rangle& \ldots & \langle Y_{rr}, B_{\alpha} \rangle \end{bmatrix} x^{\alpha},
%        \end{aligned}
%    \end{equation*}
 \begin{equation*}
        P(x) = \sum_{\alpha \in \mathbb{N}^n_{2d}} \begin{bmatrix} \langle Y_{11}, B_{\alpha} \rangle &  \ldots & \langle Y_{1r}, B_{\alpha} \rangle \\
                            \vdots & \ddots & \vdots \\
                            \langle Y_{r1}, B_{\alpha} \rangle& \ldots & \langle Y_{rr}, B_{\alpha} \rangle \end{bmatrix} x^{\alpha},
    \end{equation*}
    where $Y_{ij} \in \mathbb{S}^{N}$, $i,j = 1,\ldots,r$ is the $(i,j)$-th block of matrix $Y \in \mathbb{S}^{l}_+$. Then, the equality constraints in~\eqref{E:matrixSOSprogram} require
    \begin{equation} \label{E:MatrixSOSLinearEqualtiy}
        C_{\alpha}(u)= \begin{bmatrix} \langle Y_{11}, B_{\alpha} \rangle &  \ldots & \langle Y_{r1}, B_{\alpha} \rangle \\
                            \vdots & \ddots & \vdots \\
                            \langle Y_{r1}, B_{\alpha} \rangle& \ldots & \langle Y_{rr}, B_{\alpha} \rangle \end{bmatrix}, \quad \forall \alpha \in \mathbb{N}^n_{2d}.
    \end{equation}
    %where $C_{\alpha} := C_{0,\alpha} - \sum_{h=1}^tu_h C_{h,\alpha}$.
    Upon vectorization, this set of affine equalities can be written compactly as
    \begin{equation}
    \label{eq:matrixSOSconicConstraints}
    \begin{bmatrix}A_1 & A_2 \end{bmatrix}
    \begin{bmatrix}u \\ \text{vec}(Y) \end{bmatrix} = b
    \end{equation}
    for suitably defined matrices $A_1$, $A_2$ and a vector $b$.

    The matrix $A_1$ depends on the matrices $C_\alpha(u)$, and generally has no particular structure. Instead, $A_2$ has orthogonal rows, hence $A_2A_2^\tr$ is diagonal. To see this, let $e_i \in \mathbb{R}^r$ be the standard unit vector in the $i$-th direction and define
%    $$e_i := \begin{bmatrix} 0, \ldots, 0, 1, 0, \ldots, 0 \end{bmatrix}^\tr  \in \mathbb{R}^{r},$$
%    with $1$ in the $i$-th element, as the standard basis, and define
    $$
        E_i := e_i \otimes I_N \in \mathbb{R}^{l \times N},
    $$
    so $E_i^\tr YE_j= Y_{ij}$ selects the $(i,j)$-th $N\times N$ block of $Y$. Moreover, let $(C_{\alpha})_{ij}$ denote the $(i,j)$-th element of the matrix $C_{\alpha}$. The linear equalities~\eqref{E:MatrixSOSLinearEqualtiy} require that, for all $i,j = 1, \ldots, r$ and all $\alpha \in \mathbb{N}^n_{2d}$,
    \begin{equation}
        \langle E_i^\tr YE_j, B_{\alpha} \rangle = (C_{\alpha})_{ij}.
    \end{equation}
    Vectorization of the left-hand side yields
    $$
        \text{vec}(B_{\alpha})^\tr  (E_j^\tr  \otimes E_i^\tr ) \text{vec}(Y) = (C_{\alpha})_{ij}.
    $$
    It is then not difficult to see that the rows of the matrix $A_2$ in ~\eqref{eq:matrixSOSconicConstraints} are the vectors $\text{vec}(B_{\alpha})^\tr  \cdot (E_j^\tr  \otimes E_i^\tr )$ for all {triples} $(\alpha,i,j)$ (the precise order of the rows is not  important). To show that $A_2 A_2^\tr$ is diagonal, therefore, it suffices to show that, for any two different {triples} $(\alpha_1, i_1, j_1)$ and $(\alpha_2, i_2, j_2)$,
    %We then define a matrix $A_2 \in \mathbb{R}^{mr^2 \times l^2}$, where each row of $A_2$ is
%    $$
%        (A_2)_{\alpha,ij}
%   = \text{vec}(B_{\alpha})^\tr  \cdot (E_j^\tr  \otimes E_i^\tr ).
%    $$
%    For two different tuples $(\alpha_1, i_1, j_1)$ and $(\alpha_2, i_2, j_2)$, we know
%    \begin{equation} \label{E:MatrixRowEquality}
%        \begin{aligned}
%        &(A_2)_{\alpha_1,i_1j_1} \cdot (A_2)_{\alpha_2,i_2j_2}^\tr  \\
%        =& \text{vec}(B_{\alpha_1})^\tr  \cdot (E_{j_1}^\tr  \otimes E_{i_1}^\tr ) \cdot (E_{j_2} \otimes E_{i_2}) \cdot \text{vec}(B_{\alpha_2}) \\
%        = &  \text{vec}(B_{\alpha_1})^\tr  \cdot (E_{j_1}^\tr E_{j_2} \otimes E_{i_1}^\tr E_{i_1})  \cdot \text{vec}(B_{\alpha_2}).
%        \end{aligned} \\
%    \end{equation}
	\begin{align}
        0 &=
        \text{vec}(B_{\alpha_1})^\tr
        (E_{j_1}^\tr  \otimes E_{i_1}^\tr )
        (E_{j_2} \otimes E_{i_2})
        \text{vec}(B_{\alpha_2})
        \notag \\
        &=  \text{vec}(B_{\alpha_1})^\tr
         (E_{j_1}^\tr E_{j_2} \otimes E_{i_1}^\tr E_{i_2})
          \text{vec}(B_{\alpha_2}),
	\label{E:MatrixRowEquality}
    \end{align}
    where the second equality follows from the properties of the Kronecker product. To show~\eqref{E:MatrixRowEquality}, we invoke the properties of the Kronecker product once again to write
     \begin{subequations}
     \label{E:IndexOrthogonal}
        \begin{align}
            E_i^\tr E_j = (e_i^\tr e_j) \otimes I_N
            &= \begin{cases} I_{N}, &\text{if} \; i = j, \\
            0, & \text{otherwise},
             \end{cases} \\
            \text{vec}(B_{\alpha})^\tr  \text{vec}(B_{\beta})
            &= \begin{cases}
            n_{\alpha}, & \text{if } \alpha = \beta, \\
            0, & \text{otherwise},
             \end{cases}
        \end{align}
    \end{subequations}
    where $n_{\alpha}$ is the number of nonzeros in $B_{\alpha}$. {It is then clear that~\eqref{E:MatrixRowEquality} holds if, and in fact only if, $(\alpha_1, i_1, j_1)\neq (\alpha_2, i_2, j_2)$. Consequently, $A_2 A_2^\tr$ is diagonal.}
\end{proof}

{Proposition~\ref{theo:orthogonalityMatrixSOS} reveals an inherent structural property of SDPs derived from matrix-valued SOS programs using the monomial basis, and the algorithm of Section~\ref{Section:ADMM} applies \textit{verbatim} because the conic program representation of scalar- and matrix-valued SOS programs has the same general form.}

\begin{table*}%[t]
	\centering
	\setlength{\abovecaptionskip}{0pt}
	\setlength{\belowcaptionskip}{0em}
	\renewcommand\arraystretch{0.85}
	\caption{CPU time (in seconds) to solve the SDP relaxations of \eqref{eq:ExamplePOP}. $N$ is the size of the largest PSD cone, $m$ is the number of constraints, $t$ is the size of the matrix factorized by CDCS-sos.}
	\label{T:TimePOP}
	\begin{tabular}{c r r r r r r r r r r r r r }
		\hline \toprule[1pt] %\\[-0.75em]
		& \multicolumn{3}{c}{Dimensions} & &  \multicolumn{9}{c}{CPU time (s)} \\
		\cline{2-4} \cline{6-14}\\[-0.75em]
		%\hline \hline \\[-0.75em]
		$n$   & $N$ & $m$ & $t$ & & SeDuMi &   SDPT3  & SDPA  & CSDP  & Mosek & & SCS-direct & SCS-indirect & CDCS-sos \\
		\cline{2-4} \cline{6-10} \cline{12-14}\\[-0.5em]
		%\hline%\\[-0.75em]
		$10$ & 66 & 1\,000  & 66 & &  2.6 & 2.1& 1.6 &   2.5  & 0.8 & &  0.4  &  0.4   & 0.4   \\
		$12$ & 91 & 1\,819 & 91 & & 12.3 &    7.0&    5.7 &  4.0 & 2.4 &  &   0.7 &    0.8 &    0.7  \\
		$14$ & 120 & 3\,059 & 120 & &68.4 &   24.2&   18.1&   13.5 & 6.5 & &  1.7&   1.7 &   1.4 \\
		$17$ & 171 & 5\,984 & 171 & & 516.9 &  129.6&   97.9&   75.8& 38.1 &  &  4.6   & 4.4  &  3.5 \\
		$20$ & 231 & 10\,625& 231 & &  2\,547.4    &494.1    &452.7&   374.2  & 178.9  &   &  10.6   & 10.6 &   8.5    \\
		$24$ & 325 & 20\,474 & 325 & &**  & ** & 2\,792.8   & 2\,519.3  & 1\,398.3 & &  32.0 &   31.2   & 22.8   \\
		$29$ & 465 & 40\,919 & 465 & &** & ** & ** & ** & ** & & 125.9 &   126.3  &  67.1   \\
		$35$ & 666 & 82\,250 & 666 & &**  & **& **& **& ** &&  425.3 &   431.3 &   216.9  \\
		$42$ & 946 & 163\,184  & 946 & &** & ** & **&** & **&  &1\,415.8 &   1\,436.9  &  686.6   \\
		%$50$ & 1\,326 & 316\,250  & 1\,326  & &**  & ** & **& ** & &6\,011.2  &  6\,017.7 &   2\,223.5  \\
		\bottomrule[1pt]
	\end{tabular}%\\[0.5ex]
%	\scriptsize
%	\raggedright
%	**: the problem could not be solved due to memory limitations.
	\vspace{-3mm}
\end{table*}

%%%%%%%%%%%%%%%%%%%%%%%%%%%%%%%%%%%%%%%%%%%%%%%%%%%%%%%%%%%%
\section{Weighted SOS constraints}
\label{Section:weighted-SOS}

The discussion of Section~\ref{Section:PartialOrthogonality} is general and encompasses all SOS programs once they are recast in the form~\eqref{eq:generalSOS-multipleSOS}. As already mentioned in Section~\ref{Section:GeneralSOS}, handling SOS constraints over semialgebraic sets through~\eqref{eq:generalSOS-multipleSOS} requires introducing extra optimization variables, which is not desirable in practice. To overcome this difficulty, we show here that partial orthogonality holds also for so-called ``weighted'' SOS constraints.
%, such as those arising from applications of \textit{Positivstellensatz}~\cite{anderson2015advances}.
Specifically, consider a family of fixed polynomials $g_0,\,\ldots,\,g_t \in \mathbb{R}[x]_{n,2d}$, a second family of fixed polynomials $p_1\in \mathbb{R}[x]_{n,d_1},\,\ldots,\,p_k \in \mathbb{R}[x]_{n,d_k}$, and let $\omega_i:= \lfloor d-d_i/2\rfloor$ for each $i=1,\,\ldots,\,k$. (We have assumed that $d_1,\,\ldots,\,d_k \leq 2d$ without loss of generality.) We say that the polynomial
\begin{equation}
\label{e:g-poly}
g(x) := g_0(x) - \sum_{i=1}^t u_ig_i(x)
\end{equation}
is a weighted SOS with respect to $p_1,\,\ldots,\,p_k$ if there exist SOS polynomials $s_0 \in \Sigma[x]_{n,2d}$ and $s_i \in \Sigma[x]_{n,2\omega_i}$, $i=1,\,\ldots,\,k$, such that
\begin{equation}
\label{e:weighted-SOS}
g(x) = s_0(x) + \sum_{i=1}^k p_i(x) s_i(x).
\end{equation}
{It is not difficult to see that if $g(x)$ is a weighted SOS with respect to $p_1,\,\ldots,\,p_k$, then it is non-negative on the semialgebraic set $\mathcal{S}:=\{x \in\mathbb{R}^n :\, p_1(x)\geq 0,\,\ldots,\,p_k(x) \geq 0\}$. Thus, weighted SOS constraints arise naturally when polynomial inequalities on semialgebraic sets are cast as SOS conditions using the \textit{Positivstellensatz}~\cite{anderson2015advances}.}

To put~\eqref{e:weighted-SOS} in the form used by the standard conic program~\eqref{eq:generalSDP}, we begin by introducing Gram matrix representations for each SOS poynomial. That is, we consider matrices $X_0 \in \mathbb{S}_+^{N_0},\,X_1 \in \mathbb{S}_+^{N_1},\,\ldots,\,X_k \in \mathbb{S}_+^{N_k}$, with $N_0 := {n + d \choose d}$ and $N_i = {n + \omega_i \choose \omega_i}$ for $i=1,\,\ldots,\,k$, and rewrite~\eqref{e:weighted-SOS} as
\begin{multline}
\label{e:weighted-SOS-2}
g(x) = \langle v_d(x) v_d(x)^\tr, X_0\rangle \\
+ \sum_{i=1}^k p_i(x) \langle v_{\omega_i}(x) v_{\omega_i}(x)^\tr, X_i\rangle.
\end{multline}
In this expression, the vector $v_d(x)$ is as in~\eqref{eq:CandidateMonomials} and, similarly, $v_{\omega_i}(x)$ lists the monomials of degree no larger than $\omega_i$.

At this stage, let $B_{\alpha}$ be the {mutually orthogonal} $0/1$ indicator matrix for the monomial $x^{\alpha}$ in the outer product matrix $v_d(x)v_d(x)^\tr $, defined as in~\eqref{eq:coeff}, such that~\eqref{E:MatrixBasis} holds. Similarly, introduce symmetric {indicator} matrices $B_{\alpha}^{(i)}$ such that
\begin{equation*}
p_i(x)v_{\omega_i}(x)v_{\omega_i}^\tr(x) = \sum_{\alpha \in \mathbb{N}^n_{2d}} B_{\alpha}^{(i)} x^{\alpha}.
\end{equation*}
{Note that the matrices $B_{\alpha}^{(i)}$ are not pairwise orthogonal in general: their nonzero entries overlap to some extent because the entries of the matrix $p_i(x)v_{\omega_i}(x)v_{\omega_i}^\tr(x)$ are typically polynomials rather than simple monomials. Pairwise orthogonality holds for $B_{\alpha}^{(i)}$  if $p_i$ is a monomial, but this is uncommon in practice.}
Using such indicator matrices,~\eqref{e:weighted-SOS-2} can be written as
\begin{equation}
\label{e:weighted-SOS-3}
g(x) = \sum_{\alpha \in \mathbb{N}^n_{2d}} \left( \langle B_{\alpha}, X_0 \rangle + \sum_{i=1}^k \langle B_{\alpha}^{(i)}, X_i \rangle   \right) x^{\alpha},
\end{equation}
and we require that the coefficients of the monomials $x^\alpha$ on both sides of this expression match. To do this in compact notation, we index the monomials $x^\alpha$ using integers $1,\,\ldots,\,m$ as in Section~\ref{Section:PartialOrthogonality} and define the $m \times \sum_{i=1}^k N_i^2$ matrix
\begin{equation}
A_2 := \begin{bmatrix}
{\rm vec}(B_1^{(1)})^\tr & \cdots & {\rm vec}(B_1^{(k)})^\tr \\
\vdots & & \vdots\\
{\rm vec}(B_m^{(1)})^\tr & \cdots & {\rm vec}(B_m^{(k)})^\tr
\end{bmatrix},
\end{equation}
the $m \times N_0^2$ matrix
\begin{equation}
A_3 := \begin{bmatrix}
{\rm vec}(B_1)^\tr\\
\vdots\\
{\rm vec}(B_m)^\tr
\end{bmatrix},
\end{equation}
and the vector
\begin{equation}
\chi := \left[ {\rm vec}(X_1)^\tr,\, \cdots,\, {\rm vec}(X_k)^\tr \right]^\tr.
\end{equation}
Recalling the definition of $g(x)$ in~\eqref{e:g-poly}, we can then use the $m \times t$ matrix $A_1$ defined in~\eqref{E:VecMatrixDef} and the vector $b$ in~\eqref{e:b-vector} to write the coefficient matching conditions obtained from~\eqref{e:weighted-SOS-3} in the matrix-vector form
\begin{equation}
\label{e:weighted-sos-conic}
\begin{bmatrix} A_1 & A_2 & A_3\end{bmatrix}
\begin{bmatrix} u \\ \chi \\ {\rm vec}(X_0) \end{bmatrix}
= b.
\end{equation}

{As already noticed in Section~\ref{Section:PartialOrthogonality}, nonzero entries in $B_i$ must be zero in $B_j$ if $i\neq j$, so the rows of {$A_3$} are mutually orthogonal. Since~\eqref{e:weighted-sos-conic} corresponds to the equality constraints in the conic program formulation of a weighted SOS constraint, we obtain the following result.}
\begin{proposition}
The constraint matrix in the conic program formulation of the weighted SOS constraint~\eqref{e:weighted-SOS} has partially orthogonal rows, \emph{i.e.}, it can be partitioned as  $\begin{bmatrix} A_1 \, A_2 \, A_3\end{bmatrix}$ such that $A_3A_3^\tr $ is diagonal.
\end{proposition}

{In other words, partial orthogonality obtains also when weighted SOS constraints are dealt with directly. Thus, the ADMM algorithm descibed in Section~\ref{Section:ADMM} can in principle be applied to solve SOS programs with weighted SOS constraints. Applying the matrix inversion lemma as proposed in Section~\ref{Section:ADMM} is advantageous if $t + \sum_{i=1}^k N_i^2 < m$, meaning that the degree $\omega_1,\,\ldots,\,\omega_k$ of the SOS polynomials $s_1,\,\ldots,\,s_k$ in~\eqref{e:weighted-SOS} should be small such that
\begin{equation}
\label{e:advantage-condition-weighted-sos}
t + \sum_{i=1}^k {n+\omega_i \choose \omega_i} < {n+2d \choose 2d} =: m.
\end{equation}
Table~\ref{T:TimePOP} confirms that this is not unusual for typical problems. When~\eqref{e:advantage-condition-weighted-sos} does not hold, instead of implementing weighted SOS constraints directly, it may be more convenient introduce extra polynomials as described at the end of Section~\ref{Section:GeneralSOS}.}

\section{Numerical Experiments} \label{Section:Numerical}

We implemented the algorithm of~\cite{ODonoghue2016}, extended to take into account partial orthogonality in SOS programs, as a new package in the open-source {MATLAB} solver CDCS~\cite{CDCS}. Our implementation, which we refer to as CDCS-sos, solves step~\eqref{eq:HsdeStep1} using a sparse permuted Cholesky factorization of the matrix in~\eqref{E:HSDEFactorizationFinal}. The source code can be downloaded from \url{https://github.com/oxfordcontrol/CDCS}.
%
%\begin{center}
%\small
%\url{https://github.com/oxfordcontrol/CDCS}.
%\normalsize
%\end{center}

We tested CDCS-sos on a series of SOS programs and our scripts are available from \url{https://github.com/zhengy09/sosproblems}. %, including constrained polynomial optimizations and Lyapunov stability analysis of polynomial systems (other SOS applications can be found in the Appendix).
CPU times were compared to the direct and indirect implementations of the algorithm of~\cite{ODonoghue2016} provided by the solver SCS~\cite{scs}, referred to as SCS-direct and SCS-indirect, respectively. In our experiments, the termination tolerance for CDCS-sos and SCS was set to $10^{-3}$, and the maximum number of iterations was 2\,000. Since first-order methods only aim at computing a solution of moderate accuracy, we assessed the suboptimality of the solution returned by CDCS-sos by comparing it to an accurate solution computed with the interior-point solver SeDuMi~\cite{sturm1999using}. Besides, to demonstrate the low memory requirements of first-order algorithms, we also tested the interior-point solvers SDPT3~\cite{toh1999sdpt3}, SDPA~\cite{yamashita2003implementation}, CSDP~\cite{borchers1999csdp} and Mosek~\cite{mosek2010mosek} for comparison. {All interior-point solvers were called with their default parameters and their optimal values (when available) agree to within $10^{-8}$. All computations were carried out on a PC with a 2.8 GHz Intel\textsuperscript{\textregistered} Core\textsuperscript{\texttrademark} i7 CPU and 8GB of RAM; memory overflow is marked by ** in the tables below.}

\subsection{Constrained polynomial optimization} \label{SubSection:ConstrainedPOP}

As our first numerical experiment, we considered the constrained quartic polynomial minimization problem
\begin{equation} \label{eq:ExamplePOP}
    \begin{aligned}
    \min_{x} \quad &\sum_{1\leq i < j \leq n} (x_i x_j +x_i^2x_j - x_j^3 - x_i^2 x_j^2) \\
    \text{s. t.} \quad & \sum_{i=1}^n x_i^2 \leq 1.
    \end{aligned}
\end{equation}
We used the Lasserre relaxation of order $2d = 4$ and the parser GloptiPoly~\cite{henrion2003gloptipoly} to recast~\eqref{eq:ExamplePOP} into an SDP.

Table~\ref{T:TimePOP} reports the CPU time (in seconds) required by each of the solvers we tested to solve the SDP relaxations as the number of variables $n$ was increased. CDCS-sos is the fastest method in all cases. For large-scale POPs ($n \geq 29$), the number of constraints in the resulting SDP is over $40,000$, and all interior-point solvers (SeDuMi, SDPT3, SDPA, CSDP and Mosek) {ran out of memory on our machine}. The first-order solvers do not suffer from this limitation, and for POPs with $n \geq 29$ variables {our MATLAB solver} was approximately twice as fast as SCS. {This is remarkable considering the SCS is written in C, and} is due to the fact that $t \ll m$, cf. Table~\ref{T:TimePOP}, so the cost of the affine projection step~\eqref{eq:HsdeStep1} in CDCS-sos is greatly reduced compared to the methods implemented in SCS. {Figure~\ref{Fig:QuarticTimeAver}(a)} illustrates that, for all test problems, CDCS-sos was faster than both SCS-direct and SCS-indirect also in terms of average CPU time per 100 iterations (this metric is unaffected by differences in the termination criteria used by different solvers). Finally, Table~\ref{T:CostPOP} shows that although first-order methods only aim to provide solutions of moderate accuracy, the objective value returned by CDCS-sos and SCS was always within 0.5\% of the high-accuracy optimal value computed using interior-point solvers. Such a small difference may be considered negligible in many applications.

\begin{table}[t]
    \centering
    \setlength{\abovecaptionskip}{0pt}
    \setlength{\belowcaptionskip}{0pt}
    \caption{Terminal objective value from interior-point solvers, SCS-direct, SCS-indirect and CDCS-sos for the SDP relaxation of~\eqref{eq:ExamplePOP}.}
    \label{T:CostPOP}
    \begin{tabular}{c c c c c}
        \hline \toprule[1pt] %\\[-0.75em]
%      n  &  \textsuperscript{\textdagger}Interior-point solvers & {SCS-direct} & {SCS-indirect} & {CDCS-sos} \\
      n  &  \textsuperscript{\textdagger}Interior-point solvers & {SCS-direct} & {SCS-indirect} & {CDCS-sos} \\
         \hline %\hline \\[-0.75em]
        $10$ & $-9.11$  & $-9.12$  &  $-9.13$ & $-9.10$ \\
        $12$ & $-11.12$ & $-11.10$  &  $-11.10 $  &  $-11.11 $ \\
        $14$ & $-13.12$ & $-13.09$  &  $ -13.09$   &$ -13.12$  \\
        $17$ & $-16.12$ & $-16.09$   & $ - 16.09$ &$ - 16.06$ \\
        $20$ & $-19.12$ & $-19.17$  & $-19.17$   &$-19.08$ \\
        $24$ & $-23.12$ &  $-23.04$ & $-23.04$   &$-23.15$  \\
        $29$ & ** & $-28.17$  & $-28.18 $  & $-28.17$ \\
        $35$ & ** & $-34.05$ & $-34.05$ & $-34.08$  \\
        $42$ & ** & $-41.21$ & $-41.21$ & $-41.05$ \\
        \bottomrule[1pt]
        \end{tabular}
%        \\ \scriptsize
%\raggedright
%~\,\textsuperscript{\textdagger}: The objective values computed by SeDuMi, SDPT3, SDPA, CSDP and Mosek (when available) differ by less than $10^{-8}$.   \\
%**: The problem could not be solved due to memory limitations.
%\vspace{-1.45em}
\end{table}

\subsection{Finding Lyapunov functions} \label{SubSection:Lyapunov}

In our next numerical experiment, we considered the problem of constructing Lyapunov functions to verify local stability of polynomial systems, \emph{i.e.}, we solved the SOS relaxation of~\eqref{eq:FindLyapunov_a}-\eqref{eq:FindLyapunov_b} for different system instances. %\blue{[No: problem (8) is for global stability!]}.
We used SOSTOOLS~\cite{papachristodoulou2013sostools} to generate the corresponding SDPs.

\begin{figure}%[t]
	\centering
	%\includegraphics[scale=0.65]{Figures/Fig2_gf.eps}
%	\vspace{-0.5em}
	\includegraphics[scale=0.85]{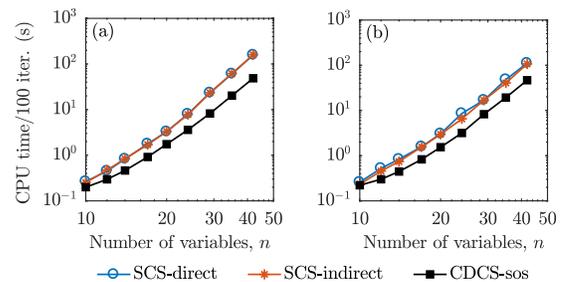}
	\caption{Average CPU time per 100 iterations for the SDP relaxations of:
		(a) the POP~\eqref{eq:ExamplePOP}; (b) the Lyapunov function search problem.}
	\label{Fig:QuarticTimeAver}
\end{figure}

\begin{table*}[t]
	\centering
	\setlength{\abovecaptionskip}{0pt}
	\setlength{\belowcaptionskip}{0em}
	\caption{CPU time (in seconds) to solve the SDP relaxations of~\eqref{eq:FindLyapunov_a}-\eqref{eq:FindLyapunov_b}. $N$ is the size of the largest PSD cone, $m$ is the number of constraints, $t$ is the size of the matrix factorized by CDCS-sos.}
	\label{T:TimeLyapunov}
	\begin{tabular}{c rrr r rrrrr r rrr }
		\hline \toprule[1pt] %\\[-0.75em]
		& \multicolumn{3}{c}{Dimensions} & &  \multicolumn{9}{c}{CPU time (s)} \\
		\cline{2-4} \cline{6-14}\\[-0.75em]
		%\hline \hline \\[-0.75em]
		$n$   & $N$ & $m$ & $t$ & & SeDuMi  & SDPT3 & SDPA  & CSDP & Mosek&  & SCS-direct & SCS-indirect & CDCS-sos \\
		%\begin{tabular}[x]{@{}c@{}}SCS\\[-0.25em]-direct\end{tabular} &
		%\begin{tabular}[x]{@{}c@{}}SCS\\[-0.25em]-indirect\end{tabular} &
		%\begin{tabular}[x]{@{}c@{}}CDCS\\[-0.25em]-sos\end{tabular} \\
		\cline{2-4} \cline{6-10} \cline{12-14}\\[-0.5em]
		%\hline%\\[-0.75em]
		$10$ & 65 & 1\,100 & 110 &
		& 2.8 &  1.8 &  2.0 &   2.6  & 0.7 &
		&  0.2  &  0.2   & 0.3   \\
		$12$ & 90 & 1\,963 & 156 &
		& 6.3 & 4.9 &   3.5 &  1.0 & 2.1 &
		&   0.3 &    0.3 &  0.4  \\
		$14$ & 119 & 3\,255 & 210 &
		& 36.2 & 16.3 &  44.8&   2.6 & 5.5 &
		&  0.8&   0.7 &   0.6 \\
		$17$ & 170 & 6\,273 & 306 &
		& 265.1 & 78.0 &   204.7&   9.5 & 26.9 &
		&  1.3   & 1.3  &  1.1 \\
		$20$ & 230 & 11\,025& 420 &
		& 1\,346.0 &  361.3      &940.5&    40.4  & 112.5&
		&  3.1   & 3.0 &   2.4    \\
		$24$ & 324 & 21\,050 & 600 &
		&** & **  & 8\,775.5    & 238.4  & 632.2&
		&  15.1 &   6.6   & 5.1   \\
		$29$ & 464 & 41\,760 & 870 &
		&** & ** & ** & ** & ** &
		& 17.1 &   16.9  &  14.3   \\
		$35$ & 665 & 83\,475 & 1260 &
		&**  & **& ** & ** & ** &
		& 67.6 &   57.1 &   37.4  \\
		$42$ & 945 & 164\,948  & 1806 &
		&** & ** & ** & ** & **&
		&133.7 &   129.2  &  92.8 \\
		\bottomrule[1pt]
	\end{tabular}%\\
%	\scriptsize
%	\raggedright
%	**: The problem could not be solved due to memory limitations.\newline
%	\vspace{-4mm}
\end{table*}

In the experiment, we randomly generated polynomial dynamical systems $\dot{x} = f(x)$ of degree three with a {linearly} stable equilibrium at the origin. We then checked for local nonlinear stability in the ball $\mathcal{D} = \{x\in\mathbb{R}^n: \sum_{i=1}^n x_i^2 \leq 0.1\}$ {using a quadratic Lyapunov function of the form $V(x) = x^\tr Q x$ and Positivstellensatz to derive SOS conditions from~\eqref{eq:FindLyapunov_a} and~\eqref{eq:FindLyapunov_b} (see \emph{e.g.},~\cite{anderson2015advances} for more details).} The total CPU time required by the solvers we tested are reported in Table~\ref{T:TimeLyapunov}, while {Figure~\ref{Fig:QuarticTimeAver}(b)} shows the average CPU times per 100 iterations for SCS and CDCS-sos. As in our previous experiment, the results clearly show that the iterations in CDCS-sos are faster than in SCS for all our random problem instances, { and that both first-order solvers have low memory requirements and are able to solve large-scale problems ($n \geq 29$) beyond the reach of interior-point solvers.}

\subsection{A practical example: Nuclear receptor signalling}

{As our last example, we considered a $37$-state model of nuclear receptor signalling with a cubic vector field and an equilibrium point at the origin~\cite[Chapter 6]{khoshnaw2015model}.} We verified its local stability within a ball of radius $0.1$ by constructing a quadratic Lyapunov function. SOSTOOLS~\cite{papachristodoulou2013sostools} was used to recast the SOS relaxation of~\eqref{eq:FindLyapunov_a}-\eqref{eq:FindLyapunov_b} as an SDP with constraint matrix of size $102\,752 \times 553\,451$ and a large PSD cone of linear dimension $741$. Such a large-scale problem is currently beyond the reach of interior-point methods on a regular desktop computer, and all of the interior point solvers we tested (SeDuMi, SDPT3, SDPA, CSDP and Mosek) ran out of memory on our machine. On the other hand, the first-order solvers CDCS-sos and SCS managed to construct a valid Lyapunov function, {with our partial-orthogonality-exploiting algorithm being more than twice as fast as SCS ($148\,$s vs. $\approx400\,$s for both SCS-direct and SCS-indirect).}

\section{Conclusion}
\label{Section:Conclusion}

In this paper, we proved that SDPs arising from SOS programs formulated using the standard monomial basis possess a structural property that we call partial orthogonality. We then demonstrated that this property can be leveraged to substantially reduce the computational cost of an ADMM algorithm for conic programs proposed in~\cite{ODonoghue2016}. Specifically, we showed that the iterates of this algorithm can be projected efficiently onto a set defined by the affine constraints of the SDP. The key idea is to exploit a ``diagonal plus low rank'' structure of a large matrix that needs to be inverted/factorized, which is a direct consequence of partial orthogonality. Numerical experiments on large-scale SOS programs demonstrate that the method proposed {in this paper
%---implemented as a new package in the open-source MATLAB solver CDCS---
yield considerable savings compared to many state-of-the-art solvers.} For this reason we expect that our method will {facilitate the} use of SOS programming for the analysis and design of large-scale systems.

\bibliographystyle{IEEEtran}
\bibliography{Reference}

\balance
\appendix

We present here a detailed count of floating-point operations (flops) to support the claims made in Remark~\ref{r:flops}. An accurate analysis that takes sparsity into account is not straightforward, especially given that sparsity is problem-dependent, so we ignore sparsity for simplicity. Following the convention in~\cite[Appendix C]{boyd2004convex}, we then take the cost of an $m\times n$ matrix-vector multiplication to be $2mn$ flops.

    To compare the complexity of our proposed method to that of SCS only the cost of solving the linear system (27) need be considered, since that is the only difference. Our method solves (27) as
    \begin{subequations}
    \begin{align}
    \sigma_1 &= \hat{\omega}_1 + {A}^\tr  \sigma_2,\\
    (I + {A}{A}^\tr ) \sigma_2 &= \hat{\omega}_2-{A} \hat{\omega}_1,
    \label{e:second-block}
    \end{align}
    \end{subequations}
    where $\hat{\omega}_1 \in \mathbb{R}^{t + N^2}$ and $\hat{\omega}_2 \in \mathbb{R}^m$ are given vectors and $A \in \mathbb{R}^{m \times (t + N^2)}$. Computing the right-hand sides for a given $\sigma_2$ cost $2m(t + N^2) + t + N^2 + (2m(t + N^2)+m)=4m(t + N^2) + t + N^2 +m$ flops, to which we have to add the cost of calculating $\sigma_2 = (I + AA^\tr )^{-1} r$ where $r= \hat{\omega}_2 -{A} \hat{\omega}_1$. Taking  advantage of the ``diagonal plus low structure'' in $I + AA^\tr$, we have
    $$
    \begin{aligned}
        (I + AA^\tr )^{-1}  &= (P + A_1A_1^\tr )^{-1} \\
    &= P^{-1} - P^{-1}A_1(I+A_1^\tr P^{-1}A_1)^{-1}A_1^\tr P^{-1},
    \end{aligned}
    $$
    where $P = I + A_2A_2^\tr$ is an $m\times m$ diagonal matrix and $A_1 \in \mathbb{R}^{m \times t}$. The inverse $P^{-1}$ and the Cholesky factorization $(I+A_1^\tr P^{-1}A_1) = LL^\tr$ can be pre-computed, so to find $\sigma_2$ at each iteration of our algorithm we need:
    \begin{itemize}
    	\item $m$ flops to compute $x = P^{-1}r$
    	\item $2mt$ flops to compute $y = A_1^\tr x$
    	\item $2t^2$ flops to solve $LL^\tr z = y$ using forward and backward substitutions
    	\item $2mt + m$ to compute $v = P^{-1}A_1z$
    	\item $m$ flops to compute $\sigma_2 = x - v$.
    \end{itemize}
    In total, therefore the proposed method requires $2t^2 + 4(N^2 + 2t)m + 4m + t + N^2 = \mathcal{O}(t^2 + mN^2+mt)$ at each iteration.	In contrast:
	\begin{itemize}
		\item The method in SCS-direct uses a cached $LDL^\tr$ factorization of (27) and requires $\mathcal{O}\left((m+t+N^2)^2\right)$ flops to carry out forward and backwards substitutions.
		\item The method in SCS-indirect solves~\eqref{e:second-block} above using a conjugate gradient (CG) method, which costs a total of
$2m(2t + 2N^2 + mn_{\rm cg}) + m + t + N^2 = \mathcal{O}(n_{\rm cg}m^2 + mN^2+mt)$
flops (here, $n_{\rm cg}$ denotes the number of CG iterations).
	\end{itemize}

\end{document}